\documentclass{amsart}

\usepackage[pagebackref = false, colorlinks=true, pdfstartview=FitV, linkcolor=blue,citecolor=red, urlcolor=blue]{hyperref}

\usepackage{amsmath,amsthm,amssymb,mathrsfs}
\usepackage{color}
\usepackage{hyperref}
\usepackage{url}

\newtheorem{thm}{Theorem}[section]

\newtheorem{prop}[thm]{Proposition}
\theoremstyle{definition}

\theoremstyle{definition}
\newtheorem{defi}[thm]{Definition}
\theoremstyle{remark}
\newtheorem{rem}[thm]{Remark}

\newcommand\be{\begin{equation}}
\newcommand\ee{\end{equation}}
\newcommand\ben{\begin{enumerate}}
\newcommand\een{\end{enumerate}}
\numberwithin{equation}{section}

\newcommand{\R}{\ensuremath{{\bf R}}}
\newcommand{\C}{\ensuremath{{\bf C}}}
\newcommand{\Z}{\ensuremath{{\bf Z}}}

\newcommand{\HH}{\ensuremath{{\bf H}}}
\newcommand{\Q}{{\bf Q}}
\newcommand{\g}{{\mathfrak g}}

\newcommand{\LL}{{\mathscr L}}
\newcommand{\M}{{\mathscr M}}

\title{A note on mock automorphic forms and the BPS index}

\author{Tian An Wong}
\email{tiananw@umich.edu}
\address{University of Michigan-Dearborn, Dearborn, MI 48128, USA}



\begin{document}

\begin{abstract}
We discuss mock automorphic forms from the point of view of representation theory, that is, obtained from weak harmonic Maa{\ss} forms give rise to nontrivial $(\g,K)$-cohomology. We consider the possibility of replacing the `holomorphic' condition with `cohomological' when generalizing to general reductive groups. Such a candidate allows for growing Fourier coefficients, in contrast to automorphic forms under the Miatello-Wallach conjecture. In the second part, we provide an overview of the connection with BPS black hole counts as a physical motivation for studying mock automorphic forms. 
\end{abstract}

\maketitle


The study of mock modular forms began with Ramanujan's introduction of mock theta functions, but precisely because of their lack of modularity, they were not well-understood and their immediate applications were not clear to many. More recently, largely motivated by problems in physics, mock modular forms have become a subject of increasing importance.

We shall consider in this note mock modular forms in the context of automorphic forms on reductive groups,  motivated by certain expectations in string theory, which we describe below. We will mainly be concerned with mock modular forms defined by weak harmonic Maa{\ss} forms (Definition \ref{mmfdef}); roughly speaking, the latter are modular but not quite holomorphic, while the former are holomorphic but not quite modular. Our candidate (Definition \ref{mockautdef}) replaces `holomorphic' with `cohomological', and we show that this is consistent with the previous definition in the case of $SL_2$. Beyond mock modular forms, other important but more complicated classes such as mixed mock modular forms and mock Jacobi forms also exist, but for simplicity we have not included these in the discussion.

In Section \ref{mocksection} we discuss mock modular forms from the point of view of representation theory and a candidate mock automorphic form; the reader not interested in the physical motivations may confine themselves to this section alone.  In Section \ref{BPSsection} we review the connection betwen BPS states and automorphic forms, discuss BPS jumping loci in relation to Kudla's generating series of special cycles on Shimura varieties. and introduce a naive---though natural---definition of an $L$-function associated to the partition function of BPS states.


\section{Mock modular forms and representation theory}
\label{mocksection}

In this section we recall basic notions of mock modular forms, and their interpretation in the representation theoretic setting. Then using this framework we discuss generalizations to higher rank groups.

\subsection{The classical theory}

Many excellent surveys of mock modular forms already exist in the literature, for example \cite{Duke,Ono,Folsom,ZagBourbaki}. We only give a general overview here, as a framework for the general case. 
The notion of a weakly holomorphic Maa{\ss} form was introduced by Brunier and Funke \cite{BF}. Let $(\rho,V)$ be a finite-dimensional complex representation of a finite-index subgroup $\Gamma$ of $SL_2(\Z)$. A (weak) Maa{\ss} form of weight $k\in\Z$ and type $(\rho,V)$ is a smooth function $F:\HH\to V$ such that
\[
F(\gamma \tau) = (c\tau + d)^k\rho(\gamma)F(\tau), \qquad \forall\gamma=\begin{pmatrix}a&b\\c&d\end{pmatrix}\in\Gamma
\]
and for some $B>0$,
\[
F(\tau) = O(e^{Bv})
\]
$v\to\infty$ uniformly in $u$, where $\tau = u + iv$ (and a similar condition holds at all cusps of $\Gamma$). Moreover, we require $F$ to be an eigenfunction of the weight $k$ Laplace operator,
\[
\Delta_k = -v^2\Big(\frac{\partial^2}{\partial u^2}+\frac{\partial^2}{\partial v^2}\Big) +ikv\Big(\frac{\partial}{\partial u}+i\frac{\partial}{\partial v}\Big).
\]
and we call $F$ harmonic if the eigenvalue is 0. Harmonic Maa{\ss} forms have also been defined for half-integral weight $k$, in which case one requires $\Gamma\subset \Gamma_0(4)$.


Let $F$ be a scalar-valued harmonic Maa{\ss} form of weight $k\in \Z\backslash\{1\}$. Solving the differential equations imposed by the harmonicity, the Fourier expansion of $F$ has a unique decomposition into $F=F^++F^-$, where
\begin{align}
\label{harmonicmaass}
F^+(\tau) &= \sum_{n\gg -\infty}c^+(n)q^n,\notag\\
F^-(\tau) &= c^-(0)v^{1-k}+\sum_{n\ll\infty}c^-(n)W_k(2\pi nv)q^n
\end{align}
where $W_k$ is the real-valued incomplete Gamma function. The functions $F^+$ and $F^-$ are referred to as the holomorphic and non-holomorphic parts of $F$, respectively. See \cite[\S3]{BF} for details. 

\begin{defi}
\label{mmfdef}
We call the holomorphic part $F^+$ a {\em mock modular form}. The non-holomorphic part $F^-$ is called the {\em completion}, obtained as the non-holomorphic Eichler integral of a weight $2-k$ modular form $f$:
\[
F^-(\tau)=(i/2)^{k-1}\int_{-\bar{\tau}}^\infty(z+\tau)^{-k}\bar{f}(z)dz
\]
Also, the image of the derivative operator
\[
\xi_k(F):= 2iv^k\frac{\overline{\partial}}{\partial\bar{\tau}}(F)
\]
is a cusp form of weight $2-k$ by \cite[Proposition 3.2]{BF}, which will be called the {\em shadow} of $F$, and is proportional to $f$. See also \cite[\S7.1]{DMZ}
\end{defi}

Let $M^!_k$ be the space of weight $k$ weakly holomorphic modular forms, $\hat{\mathbb M}_k$ the space of weight $k$ weak harmonic Maa{\ss} forms, and $M_{2-k}$ the space of weight $k$ classical modular forms. They fit into an exact sequence
\[
0 \longrightarrow M^!_k \longrightarrow \hat{\mathbb M}_k \stackrel{\xi_k}{\longrightarrow} M_{2-k} \longrightarrow 0.
\]
Denoting by $\mathbb M_k$ the space of mock modular forms, then the completion of $F^+$ by $F^-$ induces an isomorphism $\mathbb M_k\simeq \hat{\mathbb M}_k$. Also note that $\xi_k = v^{2-k}\bar{L}$, where $L$ is the Maa{\ss} lowering operator.


\begin{rem}
\label{DMZmmf}
In a similar framework, Zwegers \cite{Zwe} considered mock Jacobi forms $\mu$, which transform like non-holomorphic weight $\frac12$ Jacobi forms. 
 Zwegers' $\mu$ functions can be completed again to harmonic Maa{\ss} forms, and Ramanujan's mock theta functions can be expressed in terms of $\mu$ and ordinary modular forms.

Dabholkar, Murthy, and Zagier \cite{DMZ} gave a more general definition that captures both mock modular forms and mock Jacobi forms. Namely, they define a mock modular form to be the holomorphic part of a weight $k$ real-analytic modular form whose $\bar\tau$ derivative is decomposable \cite[\S7.3]{DMZ}, that is, they belong to the space
\[
\bigoplus_{r\in\Z} \tau^r M_{k+r}\oplus \overline{M}_{l+r}.
\]
Here $M_k$ is the space of modular forms of weight $k$. But we will not consider mixed mock modular forms nor Jacobi forms in what follows.
\end{rem}

\subsection{Transition to representation theory}

To look for generalizations of mock modular forms to higher rank groups, we must translate the current setting to the language of representation theory. There are several steps required in order to make the transition. 

For the remainder of this section, we set $G=SL_2(\R)$, and $K=SO(2)$. Our Maa{\ss} form $F$ begins as a function on $\HH$. Making the identification $\HH\simeq G/K$, we first lift $F$ to $G$ by defining
\[
\tilde{F}(g) := (cz+d)^{-k}F(g(i))
\]
where as usual where $g\in G$ acts on $i$ by M\"obius transformation. Then $\tilde{F}:G\to V$ satisfies
\begin{align}
\label{automorphy}
\tilde{F}(\gamma g) &= \rho(\gamma) \tilde{F}(g), &&\forall\gamma\in\Gamma\\
\label{kfinite}
\tilde{F}(gk_\theta) &= e^{ik\theta}\tilde{F}(g), &&\forall k_\theta=\begin{pmatrix}\cos\theta & \sin\theta\\-\sin\theta&\cos\theta\end{pmatrix}\in K\\
\tilde{F}(g) &= O(e^{Bv}), 
\end{align}
for some $B>0$ uniformly in $u$ as $v\to\infty$, where $g(i) = u+iv$. Note that if instead $\tilde{F}(g) = O(v^{B})$, then $\tilde{F}$ is said to have moderate growth.


After the pioneering work of Selberg, Maa{\ss} forms can be obtained as vectors in a certain principal series representation of $G$, written as the induced representation space
\[
I(\chi_\infty)=\text{ Ind}_{B(\R)}^{G}\chi_\infty
\]
where $\chi_\infty$ is a character of $\R^\times\to \C^\times$, and $B$ the subgroup of upper triangular matrices in $G$. Also introduce a complex parameter $s\in \C$, and write $
\chi_s(g) = \chi_\infty(g) |g|^s.$ Then taking any vector $f_s$ in $I(\chi_\infty)$, we form the Eisenstein series
\[
E(s,g) = \sum_{\gamma \in B(\Z)\backslash SL_2(\Z)} f_s(\gamma g).
\]
The special case $f_s(g) = \chi_s(g)$ corresponds to the classical Eisenstein series.  This is an example of a non-holomorphic Maa{\ss} form. 


Now, viewing $I(\chi_\infty)$ as a $(\g,K)$-module, we can associate to any Maa{\ss} form a representation $\pi_\infty$ of $G$. If $F$ is a weak harmonic Maa{\ss} form, the asscociated $I(\chi_\infty)$ is a certain degenerate principal series representation. The $(\g,K)$-modules associated to integral weight weak harmonic Maa{\ss} forms have been classified by Bringmann and Kudla \cite{BK}, and also Schulze-Pillot \cite{SP} for certain half-integral weights. As was already noted by these authors, the weak growth conditions allow for indecomposable but not irreducible Harish-Chandra modules. 

Let $A(G,V;\Gamma)$ be the space of $V$-valued smooth $K_\infty$-finite functions on $G$ satisfying the property (\ref{automorphy}). It is a $(\g,K)$-module, and the lifted Maa{\ss} form $\tilde{F}$ generates a sub-$(\g,K)$-module in $A(G,V;\Gamma)$. Then comparing this with the classification of standard $(\g,K)$-modules on $G$, \cite[Theorem 5.2]{BK} gives a complete description of the nine possible $(\g,K)$-modules generated by $\tilde{F}$, and all are shown to exist.

\subsection{Generalizations}

We now turn our attention to reductive groups other than $SL_2$. In particular, we do not consider the Jacobi group, which is not reductive. 

\subsubsection{Siegel mock modular forms}
Following Westerholt-Raum \cite{WR}, define the space of harmonic weak Siegel Maa{\ss} forms to be the space of real-anlaytic functions with possible meromorphic singularities that transform like Siegel modular forms and that are mapped to non-holomorphic Saito-Kurokawa lifts under vector-valued lowering operators. The Saito-Kurokawa lift can be viewed as a particular theta lift of weight $k$ cusp forms to a cuspidal, real-analytic Siegel modular form.

Given this definition, the associated $(\g,K)$-module is then identified as the Langlands quotient of a suitable degenerate principal series representation. Then the method of \cite{BF} can be adapted to characterize weak harmonic Siegel Maa{\ss} forms \cite[Theorem 1]{WR}. Then analogous to (\ref{harmonicmaass}), \cite[\S4.3]{WR} introduces a definition of mock Siegel modular forms using a decomposition of harmonic weak Siegel Maa{\ss} forms into meromorphic and non-holomorphic parts. 

A different Siegel mock modular form has been constructed by Kudla, Rapoport, and Yang \cite{KRY} as follows: let $B$ be a rational indefinite quaternion algebra, and $\M_B$ be the moduli space over Spec$(\Z)$ of abelian surfaces $A$ with endomorphism $\iota:\mathcal O_B\to \mathrm{End}(A)$. Given $T\in\mathrm{Sym}^2(\Z)$, let $Z(T)$ be the locus on $\M_B$ such that 
\[ \phi_2^B(\tau):= \sum_{T\text{ good}} \widehat{\deg}Z(T)e(\mathrm{tr}(T\tau)) \]
for $\tau$ in the Siegel upper half plane. $T$ is called good in the sense that $Z(T)=\mathrm{Spec}R(T)$ is a 0-cycle, and $\widehat{\deg}Z(T):=\log|R(T)|$. It can be completed to the value at $s=0$ of the first derivative of a Siegel-Eisenstein series of weight $\frac32$ and genus 2,
\[ \frac{d}{ds}\Big|_{s=0}E(\tau,s,B)=\phi_2^B(\tau) + \sum_{T\text{ not good}}c(T)e(\mathrm{tr}(T\tau). \]
Thus $\phi_2^B(\tau)$ is called a Siegel mock modular form, in the sense that it is holomorphic, and can be completed to a non-holomorphic function that transforms like an automorphic form on the metaplectic group Mp$_4(\R)$.

\subsubsection{Growth conditions}

First, by the G\"{o}tzky-Koecher principle for Siegel modular form of genus $g>1$ and Hilbert modular forms, the regularity condition at cusps of the Siegel (resp. Hilbert) modular variety is a consequence of the holomorphy and automorphy conditions on the variety itself. Indeed, this principle has been generalized to PEL Shimura varieties by Lan \cite{Lan}, for all vector-valued weights and higher coherent cohomology groups.

More generally, let $G$ be a real reductive group with compact center, $K$ a maximal compact subgroup, and $\Gamma$ a discrete cofinite subgroup. Then an automorphic form on $G$ is a function $f\in C^\infty(G/\Gamma)$ such that
\begin{align}
&\dim\mathrm{span}_\C(R_Kf) <\infty \\
&\dim Z(\g)f<\infty\\
&|xf(g)|\leq C(x)||g||^B,\quad \forall x\in U(\g) \label{mwconj}
\end{align}
for some $B>0$ and constant $C(x)$. Here $R_gf(x) := f(xg)$. Then the conjecture of Miatello and Wallach \cite{MW} states that if $G$ is a semisimple reductive group of split rank $>1$ over $\R$, and $\Gamma$ an irreducible subgroup, then the condition (\ref{mwconj}) is redundant.

Miatello and Wallach show that the conjecture is true for the group $SO(n,1)$ over a nontrivial totally real extension of $\Q$. More recently, Miller and Trinh study the conjecture in the case of $SL_3$:

\begin{prop}[\cite{MT}]
\label{MT}
Given a Maa{\ss} form $F$ on $SL_3(\Z)\backslash SL_3(\R)/SO(3)$, then:
\begin{enumerate}
\item
If the Fourier series of $F$ converges absolutely, then it satisfies the Miatello-Wallach conjecture,
\item
If $F$ has at most exponential growth at cusps, then only Jacquet's Whittaker function occurs in its Whittaker expansion.
 \end{enumerate}
 \end{prop}
 Among their conclusions is that one should not expect that automorphic forms on higher rank with Fourier coefficients of faster than polynomial growth.

\subsubsection{Holomorphy and cohomology}
\label{mockaut}

Another particularity of $SL_2$ (or $GL_2$) is the notion of holomorphy. Already for $GL_3$ the associated symmetric space is not a complex manifold, hence the definition of mock modular form as the holomorphic part of a Maa{\ss} form on say, on $SL_3(\R)$ needs to be adjusted.

Now, holomorphic automorphic forms are obtained from holomorphic discrete series on $G$ when they exist. By Harish-Chandra we know that a linear connected semisimple Lie group has discrete series if and only if the rank$(G)=\text{rank}(K)$ (which is equivalent to the existence of a compact Cartan subgroup). Indeed, $Sp_{2n}(\R)$ does have holomorphic discrete series representations for all $n$, but $SL_n(\R)$ does not for $n>2$. 

On the other hand, all discrete series representations are cohomological. Recall that $(\g,K)$-cohomology is computed from the chain complexes $C^i(\g;\mathfrak k, V)=\mathrm{Hom}_\mathfrak k(\wedge^i(\g/\mathfrak k),V)$, and the $(\g,K)$-module $V$ is called cohomological if there exists a finite-dimensional representation $M_\lambda$ of $G$, or equivalently, a $(\g,K)$-module $M_\lambda$ such that $H^*(\g,K,V\otimes M_\lambda)$ is nonzero.

Thus a possible candidate for a mock automorphic form is the following decomposition:
\begin{defi}
\label{mockautdef}
Let $\pi$ be a generic\footnote{This is simply to guarantee the existence of a Whittaker model; we expect that this condition can be relaxed.} automorphic representation on $G$. Suppose that for any $F\in\pi$ there is a decomposition
\[
F =F^+ + F^-
\]
in terms of the Fourier-Whittaker expansion of $F$, where $F^+$ is generates a (possibly reducible) cohomological $(\g,K)$-module. Then $F^+$ is called a mock automorphic form, and $F^-$ the completion of $F^+$. 
\end{defi}

\begin{rem}
The definition given above allows for growing Fourier coefficients, and generalizations the notion of a mock modular form being the `holomorphic part' of a non-holomorphic function. On the other hand, it is clear that this definition is lacking in that the completion $F^-$ is not explicitly given, for example in terms of differential operators, and hence there as yet no shadow associated to it.
\end{rem}

We verify this in the usual case of weak harmonic Maa{\ss} forms on $SL_2(\R)$, namely, we show that the mock modular forms of Definition \ref{mmfdef} support cohomology. 

\begin{thm}
Let $F^+$ be a mock modular form on $G=SL_2(\R)$, and $V_{F^+}$ be the $(\g,K)$-module generated by $F^+$. Then $H^i(\g,K;V_F)\neq0$ for some $i$.
\end{thm}

\begin{proof}
 The $(\g,K)$-module generated by $F^+$ is a submodule of the space of $K$-finite vectors in $C^\infty(G)$, as property (\ref{kfinite}) is preserved. In contrast to the harmonic Maa{\ss} forms, they need not be modular with respect to $\Gamma$.
 
Let $T$ be a Cartan subgroup of $K$ and $G$, and $\delta$ be half the sum of positive roots of $T$. Since $V_{F^+}$ descends to a holomorphic line bundle on $G/T$, it will suffice to work with Dolbeault cohomology instead. Let $\Lambda$ be a Harish-Chandra parameter, and set $\lambda=\Lambda-\delta$ and $e^\lambda$ the corresponding character of $T$. To this data we associate the homogeneous complex line bundle $\LL_\lambda$ on the complex manifold $G/T$.
 
The bundle $\LL_\lambda$ admits non-zero sections, and in fact $H^{0}(G/T,\LL_\lambda)$ can be viewed as holomorphic functions on the unit disk. Taking $\lambda=0$, we obtain the a sub-bundle $\LL_{F^+}$ of $\LL_0$ generated by $F^+$, and passing to Dolbeault cohomology we get $H^{0}(G/T,\LL_{F^+})\neq 0$.
\end{proof}

\begin{rem}
Note that the more general definition of mock modular form described in Remark \ref{DMZmmf} characterizes the functions by the image of their completion under the Dolbeault operators $\bar{\partial}$, giving further evidence for the our cohomological candidate.
\end{rem}

This should also be possible with the Siegel mock modular forms of \cite{WR}.  On the other hand, from the results of Miller and Trinh we immediately deduce the following:

\begin{prop}
Let $F^+$ be a mock automorphic form on $SL_3(\R)$, with $F^-$ nonzero and $F$ a weight 0 Maa{\ss} form. Then $F^+$ is not automorphic.
\end{prop}

\begin{proof}
From Proposition \ref{MT} we know that $F$ is of moderate growth, and moreover it is a multiple of Jacquet's Whittaker function $W_\lambda$, where $\lambda$ is any spectral parameter in general position. Then since $F^-$ is nonzero by assumption, thus the Fourier expansion of $F^+$ consists of a different linear combination. It therefore must have growing Fourier coefficients, and hence not be automorphic.
\end{proof}

For higher weight Maa{\ss} forms, one may consider the finite $K=SO(3)$-types given by Wigner functions, as in \cite{BM} for example, from which one can deduce the relevant $(\g,K)$-module structures.  We also note that the only cohomological representations on $SL(3)$ arise as $SL(2)$ discrete series representations induced along the $(2,1)$-parabolic. This suggests that the proper mock automorphic forms on $SL(3)$ should arise from a similar cohomological induction of mock modular forms.


\section{BPS states and automorphic forms}
\label{BPSsection}

In this section we provide an impressionistic view of the connection between automorphic forms and black holes. We refer to \cite{DMZ} for a more detailed discussion from an analytic point of view. 

\subsection{BPS states}

Physically, we confine ourselves mainly to compactifications of a Type II string on a Calabi-Yau variety $X$, i.e., a smooth projective variety over $\C$ with trivial canonical bundle, giving rise to an effective supersymmetric gravitational theory in 4 dimensions. That is, starting with 10-dimensional Lorentzian space time, we compactify to 4-dimensions
\[
M_{10} \to \R^{1,3} \times X
\]
where $\R^{1,3}$ is Minkowski spacetime. The number of preserved supersymmetries, denoted $N$, depends on the properties of $X$. As such, we will be interested in the moduli space of Calabi-Yau varieties. 

Now, given an integral homology class  $[\gamma]\in H_3(X,\Z)$. String theory associates to $(X,\gamma)$ the space $H(\gamma)$ of BPS states of charge $\gamma$. This has not been formulated rigorously, but is expected to be a finite-dimensional graded Hermitian vector space. According to \cite[\S2]{MM}, it should be defined as follows: let $M(\gamma)$ be the moduli space of $(\Sigma, A)$ where $\Sigma\subset X$ is a smooth Lagrangian submanifold in the class $\gamma$, with $A$ a flat $U(1)$-connection on it. Then $H(\gamma)$ is the $L^2$-cohomology of $M(\gamma)$. Alternatively, in the mirror formulation, one may take $M(\gamma)$ to be the moduli space of coherent sheaves on the mirror Calabi-Yau, with Chern classes given by $\gamma$.

Then the BPS index $\Omega(\gamma) = \dim H(\gamma)$ is a function 
\[
\Omega:H^3(X,\Z) \to \Z
\]
counting the signed degeneracies of BPS black holes with charged vector $\gamma$. It provides a microscopic description of blackhole entropy by the Boltzmann formula $S(\gamma) \sim \log \Omega(\gamma)$. The BPS index are arranged into a partition function, i.e., a generating series roughly of the form
\[
Z_\mathrm{string}(\Omega) = \sum_\gamma \Omega(\gamma) q(\gamma)
\]
where the parameter $q$ is given by the relevant Fourier expansion.

\subsection{Relation to automorphic forms}

Referring to \cite[\S13.4]{FGKP} and the references therein for details, we are led to automorphic forms as in the following examples.

\subsubsection{The $N=4$ case.}\label{K3T2Siegel}

Taking first $X=K3\times T^2$, where $K3$ is a compact $K3$ surface, yielding $N=4$ supersymmetry and $\frac12$ of BPS blackholes preserving half of the supersymmetries of the theory, and respectively $\frac14$-BPS states preserving a quarter. The symmetry in this case is given by $SL_2(\Z)\times SO(6;22,\Z)$,

The $\frac12$-BPS states are given by $\gamma=(p,0)$ or $(0,q)$, and are counted by $\Omega(q,0)=d(q^2/2)$, where $d(n)$ is the usual divisor function. They are Fourier coefficients of the modular form
\[
\Delta(\tau) =\eta(\tau)^{-24} = \sum_{n=1}^\infty d(n) e(\tau)
\]
where $e(\tau)=e^{2\pi i \tau}$, $\eta$ is the Dedekind eta function, and $\tau\in\HH$. Note that this can also be viewed as counting $\frac18$-BPS states  in an $X=T^6$ and $N=8$ theory.

The $\frac14$-BPS state count $\Omega_\frac14(p,q)=d(\frac{q^2}{2},\frac{p^2}{2},pq)$ where $d(m,n,l)$ are Fourier coefficients of a meromorphic Siegel modular form
\[
\label{igusa}
\Phi_{10}(\rho,\sigma,\tau)^{-1}=\sum_{m,n,l}d(m,n,l)e(m\sigma+n\tau+l\rho)
\]
where $\Phi_{10}$ is the Igusa cusp form, the unique weight 10 cusp form on $Sp_4(\Z)$, and $(\rho,\sigma,\tau)$ are coordinates on the Siegel upper-half plane. Its Fourier-Jacobi expansion was studied extensively in \cite{DMZ}, and in which certain (mixed) mock modular forms were found, and relations to certain wall-crossing phenomena.  For example, the poles of the meromorphic Siegel modular form $1/\Phi_{10}$ from (\ref{igusa}) detect this wall-crossing: by the Fourier-Jacobi expansion
\[ \sum_{n=-1}^\infty  \psi_m(z;\tau) q^{\sigma m} \]
where the $\psi_m(z;\tau)$ are (meromorphic) Jacobi forms, which decompose into a polar and finite part, $ \psi_m = \psi^P_m + \psi^F_m$, the polar part being completely determined by the poles of $\psi_m$, and the finite part is a finite linear combination of (mixed) mock modular forms multiplied by Jacobi theta functions
\[ \psi^F_m(z;\tau) = \sum_{l \text{ mod }2m} f^*_{m,l} (\tau) \theta_{m,l}(z;\tau) \]
using the finite-part we count the degeneracies of single-centered black holes, while the polar part detects the wall crossing, and hence the two centered-black holes.

\subsubsection{The $N=2$ case.}

Now take $X$ to be a Calabi-Yau threefold, in which case $N=2$. In this case the symmetry group $G(\Z)$ is not known in general, except that it should contain $SL_2(\Z)$. The Kontsevich homological mirror symmetry conjecture implies that $\gamma\in H^3(X,\Z)$ can be viewed as a semistable object, i.e., a special Lagrangian, in the derived Fukaya category D$^b$Fuk$(X)$ carrying an action of $G(\Z)$, and that the index $\Omega$ should be identified with the generalized Donaldson-Thompson invariants of $X$. 

Wall-crossing phenomena in this case are related to jumps in the Donaldson-Thomas invariants, and known to be much more complicated. At the same time, the degeneracies do not change under wall-crossing and is expect to have (mock) modular properties. 

If $X$ is rigid in the sense that its Hodge number $h_{1,2}=0$ and CM in the sense that it has complex multiplication by the ring of integers $\mathcal O_K$ of an imaginary quadratic extension $K=\Q(\sqrt{-D}), D>0$. Then it is conjectured that 
\be
\label{SU21}
G(\Z) = SU(2,1) \cap GL_3(\mathcal O_K),
\ee
known as the Picard modular group. Further it is conjectured, that $\Omega(\gamma)$ is given by the Fourier coefficients of an automorphic form attached to the quaternion discrete series on $SU(2,1)$.

If $X$ instead has Hodge number $h_{1,1}=1$, then it is conjectured that
$G(\Z) = G_2(\Z)$
and $\Omega(\gamma)$ is given by the Fourier coefficients of an automorphic form attached to the quaternion discrete series on $G_2$. Most importantly, it was observed in \cite[Remark 13.9]{FGKP} that the physics expects the coefficients to have exponential growth,
\be
\label{expgrowth}
\Omega(\gamma) \sim \exp(\pi\sqrt{Q_4(\gamma)})
\ee
as $\gamma$ grows large, and $Q_4(\gamma)\geq0$ is a quartic polynomial in $\gamma$. The authors mention a suggestion of Stephen D. Miller that $\Omega(\gamma)$ should instead be related to an analogous mock modular form on $G_2$.

In fact, this was our point of departure for this paper. Up until now, mock modular forms have been defined for $SL_2$, $Sp_4$, and the Jacobi group $SL_2\rtimes H_3$, the latter obtained from the Fourier-Jacobi expansion of the Siegel modular forms (\ref{igusa}). 


%



\subsection{BPS jumping loci}
\label{BPSintro}

Another motivation for considering mock automorphic forms arises in the notion of BPS jumping loci, introduced in the recent papers of Kachru and Tripathy. Namely, the symmetry group $G(\Z)$ depends on the choice of $X$, referred to as the compactification, and hence the moduli space of Calabi-Yau manifolds. They study jumping behaviour of BPS state counts as one moves around the moduli space, separate from the phenomenon of wall-crossing. In \cite[(a)]{KT}, the authors focus on so-called Type 2 jumping, upper semi-continuous jumps parallel to cohomology jump loci. 



\subsection{$N=4$}

Focusing on the case $X=K3\times T^2$, one considers the moduli space of $K3$ surfaces $\M_{K3}$ and of elliptic curves $\M_T$ respectively. In $\M_{K3}$, the BPS jumping loci coincide with Noether-Lefschetz divisors, which detect jumps in the Picard rank of $K3$, i.e., the rank of the sub-lattice in $H_2(K3,\Z)$ that can be realized as a linear combination of curves.

Following Moore \cite{M2}, the attractor mechanism of Ferrara, Kallosh, and Strominger defined by $\gamma$ gives rise to a dynamical system in $\M_{K3}$ and a fixed point. This point is shown to be a Shioda-Inose K3 surface associated to an even quadratic form \cite[Corollary 4.4.1]{M2}. Moreover, the quadratic form also defines a CM elliptic curve, associated to an attractor point in $\M_{T}$. Then \cite{KT2} show that the sum over counts of these attractor black holes coincides with the mock modular form
\[
\sum_n H(n) q^n
\]
of weight $\frac32$ on the congruence subgroup $\Gamma_0(4)\subset SL_2(\Z)$, where $H(n)$ are the so-called Hurwtiz class numbers, first studied by Zagier \cite{Zag} (see also \cite{B}). Indeed, Maulik and Pandharipande \cite{MP} show that the Noether-Lefschetz divisors are certain Heegner divisors on the symmetric space associated to $O(2,19)$, which were previously shown by Borcherds \cite{Bor} to be related to Fourier coefficients of a vector-valued modular form on $SL_2(\Z)$.

More generally, the moduli space of $K3$ compactifications are described by locally symmetric spaces associated to indefinite orthogonal groups
\[
\M_{a,b} = O(a,b,\Z) \backslash O(a,b,\R) / ((O(a,\R)\times O(b,\R)).
\]
When $a=2$, these are Shimura varieties, and the Noether-Lefshcetz loci are special Shimura subvarieties \cite{KT}. The case above corresponds to $(a,b)=(2,1)$, while \cite[\S3]{KT3} concerns the case $(2,2)$, giving a mock modular form first obtained by Hirzebruch and Zagier \cite{HZ}. In the case of Shimura varieties, this aligns with the philosophy of Kudla and Millson \cite{KM}, which roughly states that a generating series of the form
\be
\label{generating}
\phi(q) = \sum_n Z_nq^n 
\ee
where $Z_n$ are certain special cycles in $\M_{a,b}$, can be shown to be modular forms valued in the Chow group CH$(\M_{a,b})$ by means of an arithmetic theta lfiting. Most relevant to our discussion are the cases where $\M_{a,b}$ is sufficiently noncompact, so that modularity fails and instead one should consider mock automorphic forms, as suggested in \cite[\S5.1]{KT3}

The moduli space, up to a discrete action, factors as moduli spaces of certain vector multiplets and hypermultiplets
\be
\M_V\times \M_H
\ee
In the specific case where $X$ is a rigid CM Calabi-Yau threefold, the symmetry group $SU(2,1)$ from (\ref{SU21}), we have again a Shimura variety, this the BPS jumping loci in this case will again be special cycles, and by the method of Kulda we again obtain a modular generating series of special cycles.

More generally, one can speculate on the (mock) automorphy of the generating series as in (\ref{generating}), where now the $Z_n$ index the BPS jumping loci, or Noether-Lefschetz divisors in the moduli space $\M$. Though more importantly, it is not yet clear what the physical implications are regarding the automorphy of these generating series themselves, rather than the usual BPS indices.

\subsection{$L$-functions and converse theorems}
\label{Lfunsection}

Finally, we close with a brief discussion of the question \cite[\S11]{M1}: is there a natural role for $L$-functions in BPS state counting problems? In \cite{MM} the authors study the arithmetic height of the certain attractor varieties, which can be estimated by the distribution of zeroes of certain Dirichlet $L$-functions.

Consistent with the $L$-function modular form, we introduce a perhaps more a natural, though possibly naive, $L$-function for the BPS index. Namely, given the Fourier expansion of a modular form $f(q) = \sum a(n) q^n
$ we have the $L$-function
\be
L(s,f) = \sum_{n=0}^\infty \frac{a(n)}{n^s} 
\ee
which converges for Re$(s)$ large enough. For example, in the case $N=4$ and $\frac12$-BPS state counts, we have easily
\be
L(s,\Omega) = L(s,\Delta) = \sum_{n=1}^\infty \frac{d(n)}{n^s}
\ee
the $L$-function of $\Delta$. More generally, if one expects $\Omega(\gamma)$ to be a modular with respect to $SL_2(\Z)$ or a congruence subgroup thereof, one might hope to apply the converse theorem\cite{W}, which says that if $L(s,\Omega)$ is `nice' in the sense that it (i) has meromorphic continuation to $\C$, (ii) is bounded in vertical strips, and (iii) satisfies a certain functional equation, then $Z(\Omega)$ must be modular.

In the case of Siegel modular forms of genus 2 a converse theorem also exists, due to Imai \cite{Imai}, but not for meromorphic Siegel modular forms, so that the converse theorem would be applied to the reciprocal of the generating function to give
\be
L(s,\Omega^{-1}) = L(s,\Phi_{10}).
\ee
In particular, this gives reason to look for functional equations and meromorphic continuation of the $L$-series associated to BPS indices $L(s,\Omega)$, as a possible means of proving the modularity of such generating series.

\subsection*{Acknowledgments} The author thanks Ivano Lodato for enlightening conversations regarding BPS states; and Ramesh Chandra Ammanamanchi and Yashaswika Gaur for providing the impetus for this project.


\end{document}